\documentclass[a4paper, 12pt]{amsart}
\usepackage[a4paper,hmargin=2.5cm,vmargin=2.5cm]{geometry}
\usepackage[T2A]{fontenc}
\usepackage[utf8]{inputenc}

\usepackage[russian,english]{babel}
\usepackage{amsmath,amssymb,amsthm,amsfonts,mathrsfs,amscd}
\usepackage[linkcolor=blue,citecolor=blue,colorlinks=true,urlcolor=blue]{hyperref}
\usepackage{comment}
\usepackage{tikz}
\usepackage{graphics}
\usepackage{xfrac}
\usepackage{faktor}
\usepackage{rotating}
\numberwithin{equation}{section}
\newtheorem{theorem}{Theorem}[section]
\newtheorem{lemma}[theorem]{Lemma}
\newtheorem{propos}[theorem]{Proposition}
\newtheorem{corollary}[theorem]{Corollary}
\theoremstyle{definition}
\newtheorem{definition}[theorem]{Definition}

\newtheorem{proof}{Proof}
\let\origendproof\endproof
\def\endproof{\unskip\nobreak\hskip5pt plus 1fill$\square$\origendproof}

\newtheorem{rem}[theorem]{Remark}

\def\KK{\mathbb K}
\def\LL{\mathbb L}

\def\ZZ{\mathbb Z}
\def\PP{\mathbb P}

\def\gon{\mathrm{gon}}

\def\PGL{\mathop{\rm PGL}\nolimits}
\def\GL{\mathop{\rm GL}\nolimits}
\def\PSL{\mathop{\rm PSL}\nolimits}

\def\Aut{\mathop{\rm Aut}\nolimits}
\def\Bir{\mathop{\rm Bir}\nolimits}
\def\BBir{\mathop{\rm BBir}\nolimits}
\def\Ind{\mathop{\rm Ind}\nolimits}

\def\gcd{\mathop{\rm gcd}\nolimits}

\def\dP{\mathop{\rm dP}\nolimits}
\newcommand{\FF}{\mathbb{F}}
\newcommand{\CB}{\mathrm{CB}}

\begin{document}
\title{Jordan property of birational automorphism groups of surfaces and birational permutations}
\author{Alexandr Zaitsev}
\address{National research university ''Higher school of economics'', Laboratory of algebraic geometry, 6 Usacheva str., Moscow, 119048, Russia}
\email{\href{alvlzaitsev1@gmail.com}{alvlzaitsev1@gmail.com}}
\thanks{The study has been funded within the framework of the HSE University Basic Research Program.} 

\maketitle
\begin{abstract}
   We prove that the group of birational automorphisms of a geometrically irreducible algebraic surface over a finite field is Jordan. We show that the analogous statement fails in higher dimensions. Finally, we prove that groups of birational permutations over finite fields have bounded finite $p'$-subgroups; in particular, they are~$p$-Jordan.
\end{abstract}

\tableofcontents

\section{Introduction}

Groups of birational automorphisms of algebraic varieties are often complicated. One of the possible ways to study such groups is to look at their finite subgroups. In particular, one can study the Jordan property of infinite groups.

Recall that a group $\Gamma$ is called \emph{Jordan} if there exists a constant $J=J(\Gamma)$ such that every finite subgroup $G \subset \Gamma$ contains a normal abelian subgroup of index at most $J$. The minimal such $J$ is called a \emph{Jordan constant} of $\Gamma$.

Over fields of characteristic zero, the Jordan property for groups of birational automorphisms is by now well understood in many cases. In particular, the group of birational automorphisms of projective plane $\Bir(\PP^2)$ is Jordan (see~\cite[Theor\'em\`e 3.1]{Serre}). More
generally, in~\cite{PrShr higher dim} Prokhorov and Shramov proved that~$\Bir(X)$ is Jordan for every non-uniruled variety~$X$, and also showed that birational automorphism groups of rationally connected varieties of dimension $n$ are uniformly Jordan.
In particular, this gives the Jordan property for~$\Bir(\PP^n)$ over fields of
characteristic zero. The situation in positive characteristic is different. Already over an algebraically closed field~$\KK$ of characteristic~$p>0$ the group~$\PGL_2(\KK)$ contains simple subgroups isomorphic to~$\PSL_2(\FF_{p^r})$ of arbitrarily large order, so the usual Jordan property fails already for the group~$\relpenalty=10000 \Bir(\PP^1) \simeq \Aut(\PP^1) \simeq \PGL_2(\KK)$. On the other hand, this argument does not apply over finite fields.

In~\cite{main}, Prokhorov and Shramov proved that the group $\Bir(\PP^2_{\FF_{q}})$ over a finite field $\FF_q$ is Jordan. A natural question here is whether over finite fields one can strengthen this result and prove the Jordan property for birational automorphism groups of all geometrically irreducible surfaces. The purpose of the first part of this paper is to answer this question. Namely, the main result of the paper is the following theorem.
\begin{theorem}\label{theo: main}
Let $S$ be a geometrically irreducible algebraic surface over $\FF_q$. Then the group $\Bir(S)$ is Jordan.
\end{theorem}
The key point is to treat surfaces of Kodaira dimension $-\infty$ that are not geometrically rational. After running an equivariant minimal model program, one is naturally led to conic bundles over curves. Thus, the main technical part of the paper is devoted to birational automorphism groups of two-dimensional conic bundles over a finite field $\FF_q$. More precisely, we rewrite several arguments from \cite[Sections 4 and 6]{main} in a slightly more general form, so that they apply to conic bundles.

Combining this with elementary bounds on gonality and with estimates for automorphism groups of curves over finite fields, we obtain explicit bounds in the cases when the base of a conic bundle has genus $0$, $1$, or at least $2$. In particular, we prove the following proposition.

\begin{propos}\label{prop: g>=1}
Let $\FF_q$ be a finite field. Let $\varphi \colon S \to C$ be a two-dimensional geometrically irreducible conic bundle over $\FF_q$, where $C$ is a smooth curve of genus $g \geqslant 1$. Then the group of birational automorphisms~$\Bir(S)$ contains a normal abelian subgroup of index at most
\begin{enumerate}
    \item $J_{\CB}^1 = 24(1 + q + 2\sqrt{q})\max\{q^2(q^4-1), 60\}$ if $g = 1$;
    \item $J_{\CB}^{\geqslant 2} = 75g^4\max\{q^{2g-2}(q^{4g-4}-1),60\}$ if $g \geqslant 2$.
\end{enumerate}
\end{propos}
It is worth noting that these explicit estimates are also of interest in connection with research on the Jordan constants of groups of birational automorphisms (see~\cite{Yas} and~\cite{Zai} in characteristic $0$, also see~\cite{main} and~\cite{Vikulova} in case of finite fields). Finally, together with the results of \cite{surfaces} on surfaces of non-negative Kodaira dimension, this yields Theorem \ref{theo: main}. We conclude this part by showing that the analogous statement does not extend to higher dimensions and the group of birational automorphisms of~$\PP^n$, $n > 2$, is not Jordan over any finite field.

\begin{theorem}\label{theo: dim3}
For every finite field $\FF_q$ and for any integer $n > 2$, the group $\Bir(\PP^n_{\FF_q})$ is not Jordan.
\end{theorem}

In the second part, we establish Jordan-type properties for groups of birational permutations $\BBir(X)$ of the variety $X$ over a finite field $\FF_q$. Recall that the group~$\BBir(X)$ is a subgroup of $\Bir(X)$, which consists of birational automorphisms $f$ such that~$f$ and~$f^{-1}$ have no $\FF_q$-points of indeterminacy. Over finite fields this group was studied by Cantat in~\cite{Can09} and by Asgarli, Lai, Nakahara, Zimmermann in~\cite{ALNZ22}. In positive characteristic there are natural analogies for Jordan property and for boundedness of finite subgroups (for definitions see Section~\ref{sect: birational permutations}). For example, Brauer and Feit, and later Larsen and Pink, proved that over any
field $\KK$ of positive characteristic $p$ the groups~$\GL_n(\KK)$ are~$p$-Jordan; see~\cite[Theorem~0.4]{LP11}. Also,~\hbox{$p$-Jordan} property was established for algebraic groups by Hu in~\cite{Hu20} and for $\Bir(\PP^2)$ by Chen and Shramov in~\cite{surfaces}.
%In the first section we prove the conic bundle estimate. In the second section we derive corollaries for base curves of various genera. In the final section we combine these results with the existing statements on birational automorphism groups of surfaces in positive characteristic and prove Theorem \ref{main}.
We prove the following theorem.
\begin{theorem}\label{theo: BBir}
    Let $\FF_q$ be a finite field of characteristic $p$. Let $X$ be an irreducible variety over $\FF_q$ with a $\FF_q$-point. Then $\BBir(X)$ has bounded $p'$-subgroups. In particular, $\BBir(X)$ is $p$-Jordan.  
\end{theorem}

The paper is organized as follows. In Section~\ref{sect: conic bundles} we prove a general statement on
birational automorphism groups of two-dimensional conic bundles over finite fields. In Section~\ref{sect: Jordan property for surfaces} we apply this result together with bounds on gonality and automorphism groups of curves to prove Theorem~\ref{theo: main}. In Section~\ref{sect: higher dimensions} we show that the analogous statement does not extend to higher dimensions by proving Theorem~\ref{theo: dim3}. Finally, in Section~\ref{sect: birational permutations} we turn to groups of birational permutations over finite fields and prove Theorem~\ref{theo: BBir}.

All varieties in the text are assumed to be projective. We will use the following notation. For a finite field we write $\FF_q$, where
$q=p^l$ and~$p$ is the characteristic. By~$\overline{\KK}$ we denote the algebraic closure of the field~$\KK$. If $\KK \subset \LL$ is an extension of fields, and $X$ is a variety over~$\KK$, then by $X_{\LL}$ we denote the extension of scalars of $X$ to $\LL$.
For an irreducible algebraic variety $X$ over a field~$\KK$, we denote by~$\KK(X)$ its field of rational functions. Also we denote by $\Bir(X)$ the group of birational automorphisms of $X$ and by $\Aut(X)$ the group of regular automorphisms of $X$. Let~$\pi \colon X \xrightarrow{} B$ be a morphism of algebraic varieties; then by $\Bir(X,\pi)$ we denote the group that consists of all birational automorphisms of~$X$ mapping the fibers of $\pi$ again to the fibers of $\pi$. 

\textbf{Acknowledgements.} I would like to thank my advisor Constantin Shramov for stating the problem and useful discussions.

\section{Conic bundles}\label{sect: conic bundles}

All the results of this section are taken from \cite[Section 4, Section 6]{main}. We merely rewrite them slightly, so that they are applicable in the conic bundle setting.

\begin{lemma} \label{elementOrder}
Let $\KK \supset \FF_q$ be a purely transcendental field extension, and let $\LL \supset \KK$ be a field extension of finite degree $k$. Suppose that the multiplicative group $\LL^*$ contains an element $\zeta$ of finite order $m$. Then $m$ divides $q^k - 1$. 
\end{lemma}

\begin{proof}
See \cite[Lemma 4.2]{main}.
\end{proof}

\begin{lemma} \label{autOrder}
Let $\KK \supset \FF_q$ be a purely transcendental field extension, let $\LL \supset \KK$ be a field extension of degree $k$, and let $Q$ be a conic over $\LL$. Suppose that the group $\Aut(Q)$ contains an element of finite order $m$ coprime to $q$. Then $m$ divides $q^{2k} - 1$.
\end{lemma}

\begin{proof}
Let $g \in \Aut(Q)$ be an element of order $m$. Since $m$ is coprime to $q$, we conclude that $g$ is a semi-simple element in
$$\Aut(Q_{\overline{\LL}}) \cong \PGL_2(\overline{\LL}).$$
Hence $g$ has exactly two fixed points on $Q_{\overline{\LL}} \cong \PP^1_{\overline{\LL}}$. Thus there exists a quadratic extension~$\relpenalty = 10000 \LL' \supset \LL$ such that $g$ has a fixed point $P$ on $Q_{\LL'}$. Therefore $g$ acts faithfully on the Zariski tangent space $T_P(Q_{\LL'}) \cong \LL'$, so it is an element of finite order of $(\LL')^*$. As $\LL' \supset \KK$ is an extension of order $2k$, the assertion follows from Lemma \ref{elementOrder}.
\end{proof}

\begin{lemma} \label{PSLSize}
Let $\KK \supset \FF_q$ be a purely transcendental field extension, let $\LL \supset \KK$ be a field extension of degree $k$, and let $Q$ be a conic over $\LL$. Then the group $\Aut(Q)$ does not contain subgroups isomorphic to $\PSL_2(\FF_{p^r})$ for $p^r > q^k$.
\end{lemma}

\begin{proof}
Let $q = p^l$. By \cite[Lemma 2.3]{main} the group $\PSL_2(\FF_{p^r})$ contains an element whose order equals $\frac{p^r+1}{2}$ if $p$ is odd and an element whose order equals $2^r + 1$ if $p = 2$. On the other hand, by Lemma \ref{autOrder} the finite orders of elements of the group $\Aut(Q)$ that are coprime to $p$ divide $q^{2k} - 1 = p^{2kl} - 1$.

If $r > kl$, then we can compute that $p^{2kl} - 1$ is not divisible by $\frac{p^r + 1}{2}$ in the case of odd~$p$, and by $2^r + 1$ in the case of $p = 2$ (see the proof of \cite[Lemma 4.4]{main}).
\end{proof}

We will need the following classification.

\begin{theorem}[{see, for example, \cite[Theorem 2.1]{DD}}]\label{PGL}
Let $\LL$ be a field of characteristic $p$, and let $G$ be a finite subgroup of $\PGL_2(\LL)$. Then $G$ is isomorphic to one of the following groups:
\begin{enumerate}
\setlength\itemsep{-1 pt}
\item a dihedral group of order $2m$, $m \geqslant 2$;
\item one of the groups $A_4$, $S_4$ or $A_5$;
\item the group $\PSL_2(\FF_{p^r})$ for some $r \geqslant 1$;
\item the group $\PGL_2(\FF_{p^r})$ for some $r \geqslant 1$;
\item a group of the form $G_p \rtimes \ZZ/m\ZZ$, where $m \geqslant 1$ is coprime to $p$ and $G_p$ is an abelian $p$-group.
\end{enumerate}
\end{theorem}

\begin{lemma} \label{char}
Let $\KK \supset \FF_q$ be a purely transcendental field extension and let $\LL \supset \KK$ be a finite extension of degree $k$. Let $Q$ be a conic over $\LL$, and let $G$ be a finite subgroup of~$\Aut(Q)$. Then $G$ contains a characteristic abelian subgroup of index at most
$$J = \max\{q^k(q^{2k} - 1), 60\}.$$
\end{lemma}

\begin{proof}
Consider the cases of Theorem \ref{PGL}. If $G$ is of type (1), then it has a cyclic characteristic subgroup of index $2$. If $G$ is of type (2) then $|G| \leqslant 60$. If $G$ is of type (3) or (4), then $$|G| \leqslant |\PGL_2(\FF_{q^k})| = q^k(q^{2k} - 1)$$
by Lemma \ref{PSLSize}.
If $G$ is of type (5), then there is a unique Sylow $p$-subgroup $G_{(p)}$ in $G$, which is therefore characteristic; the index of $G_{(p)}$ in $G$ is at most $q^{2k} - 1$ by Lemma \ref{autOrder}. Taking the maximum we obtain the required value of $J$.
\end{proof}

\section{Jordan property for surfaces}\label{sect: Jordan property for surfaces}

In this section we combine together the results of the previous section to prove the existence of normal abelian subgroups of finite and bounded index in the group of birational automorphisms of a conic bundle.

Let us start with recalling the standard definition. 

\begin{definition}\label{def: gonality}
    Let $C$ be an algebraic curve over a field $\KK$. The $\emph{gonality}$ of $C$ is the lowest degree of a non-constant rational map from $C$ to $\PP^1_{\KK}$. We denote this by $\gon(C)$. Equivalently, one can define the gonality of $C$ by the following formula  
    \[
    \gon(C) = \min_{f \in \KK(C)}[\KK(C):\KK(f)].
    \]
\end{definition}

Also, recall the Hasse--Weil bound.

\begin{theorem}[{see \cite[Corollaire 3]{Weil}}]\label{theo: Hasse - Weil}
Let $C$ be a smooth projective geometrically irreducible curve of genus $g$
over a finite field $\FF_q$ and $N$ be the number of its $\FF_q$-points. Then
\[
\bigl|N-(q+1)\bigr|\leqslant 2g\sqrt q.
\]
Equivalently,
\[
q+1-2g\sqrt q \leqslant N \leqslant q+1+2g\sqrt q.
\]
\end{theorem}

\begin{lemma} \label{gonality}
Let $C$ be a smooth geometrically irreducible curve of genus $g$ over a finite field $\FF_q$. Then its gonality does not exceed
\begin{itemize}
\setlength\itemsep{-1 pt}
\item $\gon(C) = 1$ if $g = 0$;
\item $\gon(C) = 2$ if $g = 1$;
\item $\gon(C) \leqslant 2g - 2$ if $g \geqslant 2$.
\end{itemize}
\end{lemma}

\begin{proof}
If $g = 0$, then $C$ is a conic. Then $C$ has a $\FF_q$-point by Chevalley theorem, which means that $C \simeq \PP_{\FF_q}^1$ and its gonality is equal to $1$. If $g = 1$, then its gonality automatically is greater than $1$. Applying Theorem~\ref{theo: Hasse - Weil}, we obtain that the number of points on $C$ is greater or equal to $$q + 1 - 2g\sqrt{q} = q + 1 - 2\sqrt{q} = (1 - \sqrt{q})^2 > 0.$$ In particular $C$ contains a $\FF_q$-point $P$, the linear system $|2P|$ gives us a double cover of~$\PP^1$ and $\gon(C) = 2$.

Consider $g \geqslant 2$. If $C$ is hyperelliptic, then $\gon(C) = 2 \leq 2g -2$. If $C$ is not hyperelliptic, then the canonical system $K_C$ gives us an embedding $f \colon C \to \PP^{g - 1}$ such that the image has degree $2g - 2$. We can now project it to $\PP^1$ from a subspace of codimension $2$ and get a $(2g-2)$-cover of $\PP^1$, which implies $\gon(C) \leqslant 2g-2$.
\end{proof}

\begin{theorem}\label{theo: main1}
Let $\FF_q$ be a finite field. Let $\varphi \colon S \to C$ be a two-dimensional conic bundle over $\FF_q$, where $C$ is a curve of gonality $k$. Then~$\Bir(S, \varphi)$ contains a normal abelian subgroup of index at most
$$J_{\CB}(C) = |\Aut(C)|\cdot\max\{q^k(q^{2k}-1), 60\}.$$
\end{theorem}

\begin{proof}
The group $\Bir(S, \varphi)$ fits into the exact sequence
$$1 \to \Bir_\varphi(S) \to \Bir(S, \varphi) \to \Aut(C),$$
since birational automorphisms of smooth curves are regular (see, for example,~\hbox{\cite[Corollary II.2.4.1]{Sil}}). So the index of $\Bir_\varphi(S)$ in $\Bir(S, \varphi)$ is at most $|\Aut(C)|$.
On the other hand, $\Bir_\varphi(S)$ is a subgroup of the automorphism group of the scheme-theoretic generic fiber~$Q$ of $\varphi$. Observe that $Q$ is a conic over $\FF_q(C)$, which is a degree~$k$ extension of purely transcendental extension $\KK \supset \FF_q$. So by Lemma \ref{char} the group $\Bir_\varphi(S)$ contains a characteristic abelian subgroup $A$ of index at most $$J = \max\{q^k(q^{2k}-1), 60\}.$$
Therefore $A$ is a normal abelian subgroup of $\Bir(S, \varphi)$ and its index does not exceed 
\[
J_{\CB} = J \cdot |\Aut(C)|.
\]
\end{proof}

Before proving Proposition~\ref{prop: g>=1} we state a theorem proved by Stichtenoth (see~\cite{Sti73}), providing an upper bound on the automorphism group of curves of genus $g \geqslant 2$ over fields of positive characteristic.

\begin{theorem}[Stichtenoth~\cite{Sti73}]\label{theo: Stichtenoth}
Let $C$ be a genus $g \geqslant 2$ irreducible curve defined over a field~$\KK$ of characteristic $p > 0$. Then
$$|\Aut(C)| < 16g^4,$$
unless $C$ is the curve with equation
$$y^{p^n} + y = x^{p^{n} + 1},$$
in which case it has genus $g = \frac{1}{2}p^n(p^n - 1)$ and $|\Aut(C)| = p^{3n}(p^{3n} + 1)(p^{2n} - 1)$.
\end{theorem}

\begin{corollary}\label{cor: |Aut| <= 75g^4}
    Let $C$ be a genus $g \geqslant 2$ irreducible curve defined over a field~$\KK$ of characteristic $p > 0$. Then $$|\Aut(C)| < 75g^4.$$
\end{corollary}

\begin{proof}
    By Theorem~\ref{theo: Stichtenoth} we either have $|\Aut(C)| < 16g^4$ or $C$ is isomorphic to the curve~$C_n$ for some $n \in \ZZ_{\geqslant 1}$ with 
    \[
    g(C_n) = \frac{1}{2}p^n(p^n - 1), \quad |\Aut(C_n)| = p^{3n}(p^{3n} + 1)(p^{2n} - 1).   
    \]
    In the first case we are done. Let us bound the value $\frac{|\Aut(C)|}{g^4}$ in the second case. Consider a rational function
    \[
    f(x) = \frac{x^3(x^3+1)(x^2 - 1)}{(\frac{1}{2}x(x-1))^4} = 16\frac{(x+1)^2(x^2-x+1)}{x(x-1)^3}.
    \]
    Note that 
    \[
    f'(x) \leqslant 0 \;\text{for}\; x\in(1, +\infty), \quad \lim_{x \rightarrow +\infty}f(x) = 16, \quad f(p^n) = \frac{|\Aut(C_n)|}{g(C_n)^4}.
    \]
    Therefore, the value $\frac{|\Aut(C)|}{g^4}$ attains its maximum in the case $p^n=3$, i.e. for $p=3$ and~$n = 1$ (for $p=2$ curve $C_1$ has genus $1$) and this value is $\frac{|\Aut(C_1)|}{g(C_1)^4} = \frac{224}{3} < 75$.
\end{proof}

Now we are able to prove Proposition~\ref{prop: g>=1}.

\begin{proof}[{of Proposition~\ref{prop: g>=1}}]
First, note that $\Bir(S) = \Bir(S,\varphi)$ since there are no non-trivial morphisms from a conic to a curve of genus~$\relpenalty=10000 g \geqslant 1$. Applying Theorem \ref{theo: main1} we obtain that~$\Bir(S)$ contains a normal abelian subgroup of index at most $$J_\CB(C) = |\Aut(C)|\max\{q^{\gon(C)}(q^{2\gon(C)}-1), 60\}.$$ 

If $g = 1$, then we have $\gon(C) = 2$ by Lemma~\ref{gonality}. To bound the group $|\Aut(C)|$ in this case, note that the automorphisms of elliptic curve are generated by translations and the automorphisms with a fixed point. The number of translations is equal to the number of points on the curve and the number of automorphisms that preserve the given point is less or equal to $24$ (see~\hbox{\cite[Theorem III.10.1]{Sil}}). Hence, applying Theorem~\ref{theo: Hasse - Weil} we get $$|\Aut(C)| \leqslant 24(1 + q + 2\sqrt{q}).$$ This proves the first assertion.

If $g \geqslant 2$, then we have $\gon(C) \leqslant 2g-2 $ by Lemma \ref{gonality}. Applying Corollary~\ref{cor: |Aut| <= 75g^4}, we obtain the second assertion.
\end{proof}

Also we need a bound for conic bundles with base of genus $0$.

\begin{lemma} \label{lemma: g=0}
Let $\FF_q$ be a finite field. Let $\varphi \colon S \to C$ be a two-dimensional geometrically irreducible conic bundle over $\FF_q$, where $C$ is a smooth curve of genus $0$. Let $\Bir(S, \varphi)$ be the group that consists of all birational automorphisms of $S$ mapping the fibers of $\varphi$ again to the fibers of $\varphi$. Then~$\Bir(S, \varphi)$ contains a normal abelian subgroup of index at most~$\relpenalty=10000 J_{\CB}^0 = \binoppenalty = 10000 q(q^2 - 1)\max\{q(q^2-1), 60\}.$
\end{lemma}

\begin{proof}
Curve $C$ is a conic. Therefore, $C$ has a $\FF_q$-point by Chevalley theorem, which means that $C \simeq \PP^1_{\FF_q}$ and $$|\Aut(C)| = |\Aut(\PP_{\FF_q}^1)| = |\PGL_2(\FF_q)| = q(q^2 - 1).$$ Applying Theorem \ref{theo: main1} we obtain the result.
\end{proof}

Finally, let us recall results on del Pezzo surfaces from~\cite{main} and on surfaces of non-negative Kodaira dimension from~\cite{surfaces}.

\begin{propos}[{\cite[Proposition 5.6]{main}}]\label{prop: del Pezzo}
    Let $S$ be a del Pezzo surface over $\FF_q$. Then~$\Aut(S)$ contains a normal abelian subgroup of index at most
    \[
    J_{\dP} = 
    \begin{cases}
        q^3(q^2 - 1)(q^3 - 1), \quad \text{if } q \text{ is odd};\\
        \max\{q^3(q^2 - 1)(q^3 - 1), |W(E_8)|\} \quad \text{if } q \text{ is even}.\\
    \end{cases}
    \]
\end{propos}

\begin{theorem}[{\cite[Theorem 1.7(iii)]{surfaces}}]\label{theo: kappa >= 0}
     Let $\KK$ be an algebraically closed field of characteristic $p > 0$, and
let $S$ be an irreducible algebraic surface of non-negative Kodaira dimension over $\KK$. Then the group $\Bir(S)$ is Jordan.
\end{theorem}

Now we are able to prove Theorem~\ref{theo: main}.

\begin{proof}[{of Theorem~\ref{theo: main}}]
If $S$ has non-negative Kodaira dimension, then $\Bir(S_{\overline{\FF}_q})$ is Jordan by Theorem~\ref{theo: kappa >= 0} and, therefore, $\Bir(S)$ is Jordan. Otherwise, put $$J = \max\{J^0_{\CB},J^1_{\CB},J^{\geqslant 2}_{\CB}, J_{\dP}\}$$
and consider a finite subgroup~$G \subset \Bir(S)$. Let~$\tilde{S}$ be a surface birationally isomorphic to $S$, where $G$ acts regularly (see~\mbox{\cite[Lemma 3.5]{D-I}}), and let $S'$ be a result of~$G$-Minimal Model Program applied to~$\tilde{S}$. Then $S'$ is either a del Pezzo surface, or has a structure of a conic bundle~$\varphi\colon S' \rightarrow C$ with~$G \subset \Aut(S',\varphi)$. In the former case $G$ has a normal abelian subgroup of index at most $J$ by Proposition~\ref{prop: del Pezzo}. In the latter case $G$ has a normal abelian subgroup of index at most~$J$ by Proposition~\ref{prop: g>=1} and Lemma~\ref{lemma: g=0}. This proves the Jordan property for $\Bir(S)$.
\end{proof}

\section{Higher dimensions}\label{sect: higher dimensions}

In this section, we show that the result of Theorem~\ref{theo: main} does not hold in higher dimensions proving Theorem~\ref{theo: dim3}. We need the following auxiliary lemma.

\begin{lemma}\label{lemma: for PGL_3(F_q(t))}
    The group $\PGL_3(\FF_q(t))$ is not Jordan.
\end{lemma}

\begin{proof}
    
For a fixed $n \in \mathbb{N}$ put $m = \lfloor \frac{n}{2} \rfloor$ and consider a subgroup $G_n \subset \PGL_3(\FF_q(t))$ of matrices of the form
$$
\begin{pmatrix}
1& a& b\\
0& 1& c\\
0& 0& 1\end{pmatrix},$$
for $b \in \FF_q[t]_{\leqslant n}$ and $a,\, c \in \FF_q[t]_{\leqslant m}$, where $\FF_q[t]_{\leqslant n}$ and $\FF_q[t]_{\leqslant m}$ are sets of polynomials of degree less or equal to $n$ and $m$, respectively. In particular, we have a formula for the order:
\[
|G_n| = |\FF_q[t]_{\leqslant n}|\cdot|\FF_q[t]_{\leqslant m}|^2 = q^{n+1}\cdot q^{2m+2}.
\]
The multiplication works as follows
$$
\begin{pmatrix}
1& a_1& b_1\\
0& 1& c_1\\
0& 0& 1\end{pmatrix}
\begin{pmatrix}
1& a_2& b_2\\
0& 1& c_2\\
0& 0& 1\end{pmatrix} = 
\begin{pmatrix}
1& a_1+a_2& b_1 + b_2 + a_1c_2\\
0& 1& c_1 + c_2\\
0& 0& 1\end{pmatrix}.$$
The inverse of an element is given by
$$\begin{pmatrix}
1& a& b\\
0& 1& c\\
0& 0& 1\end{pmatrix}^{-1} = 
\begin{pmatrix}
1& -a& -b + ac\\
0& 1& -c\\
0& 0& 1\end{pmatrix},$$
so it is actually a subgroup.
Also note that the center $Z(G_n)$ is given by $a=c=0$, in particular, we have 
\[
|Z(G_n)| = |\FF_q[t]_{\leqslant n}| = q^{n+1}.
\]

Let us prove that there are no abelian subgroups of $G_n$ of index smaller than $q^{m+1}$. Let~$A \subset G_n$ be an abelian subgroup. If $A \subset Z(G_n)$, then 
\[
[G_n:A] \geqslant [G_n:Z(G_n)] = |\FF_q[t]_{\leqslant m}|^2 = q^{2m+2} > q^{m+1}.
\]
If $A \not\subset Z(G_n)$, then consider an element 
\[g =
\begin{pmatrix}
1& a_g& b_g\\
0& 1& c_g\\
0& 0& 1\end{pmatrix} \in A \backslash Z(G_n)
\]
and bound the order of its centralizer $C_{G_n}(g)$. The centralizer $C_{G_n}(g)$ consists of matrices 
\[
\begin{pmatrix}
1& x& y\\
0& 1& z\\
0& 0& 1\end{pmatrix}
\]
with property $a_gz = c_gx$. Since $g \not\in Z(G_n)$, then $(a_g,c_g) \neq (0,0)$. If $a_g = 0$ or $c_g = 0$, then we have $x=0$ or $z=0$, respectively, and 
\[
|C_{G_n}(g)| \leqslant |\FF_q[t]_{\leqslant n}|\cdot|\FF_q[t]_{\leqslant m}| =q^{n+1}\cdot q^{m+1}
\]
in both cases. If $a_g \neq 0$ and $c_g \neq 0$, then put $\tilde{a}_g = \frac{a_g}{\gcd(a_g,c_g)}$ and  $\tilde{c}_g = \frac{c_g}{\gcd(a_g,c_g)}$. Therefore, condition $a_gz = c_gx$ is equivalent to condition $\tilde{a}_gz = \tilde{c}_gx$ and we get $x = \tilde{a}_gw$, $z = \tilde{c}_gw$ for some $w \in \FF_q[t]_{\leqslant m}$. Hence, the number of possible pairs $(x,z)$ is bounded by $|\FF_q[t]_{\leqslant m}|$ and
\[
|C_{G_n}(g)| \leqslant |\FF_q[t]_{\leqslant n}|\cdot|\FF_q[t]_{\leqslant m}| = q^{n+1}\cdot q^{m+1}
\]
in this case too.

Since $A \subset C_{G_n}(g)$, we obtain inequalities 
\[
[G_n:A] \geqslant [G_n:C_{G_n}(g)] = \frac{|G_n|}{|C_{G_n}(g)|} \geqslant \frac{q^{n+1} \cdot q^{2m+2}}{q^{n+1}\cdot q^{m+1}} = q^{m+1}.
\]

As $n$ tends to $\infty$ the value $q^{m+1}$ also tends to $\infty$, and, therefore, group $\PGL_3(\FF_q(t))$ is not Jordan.
\end{proof}

\begin{proof}[{of Theorem~\ref{theo: dim3}}]
Since we have a birational isomorphism
\[
\PP^n_{\FF_q} \dashrightarrow \PP^2_{\FF_q} \times \PP^1_{\FF_q} \times \PP^{n-3}_{\FF_q},
\]
there exists an embedding of groups $$\Bir(\PP_{\FF_q}^2 \times \PP_{\FF_q}^1)  \hookrightarrow \Bir(\PP^n_{\FF_q}).$$ Also, $\Bir(\PP_{\FF_q}^2 \times \PP_{\FF_q}^1)$ contains~$\PGL_3(\FF_q(t))$ as an automorphism group of the generic fiber of the projection to the second factor. The latter group is not Jordan by Lemma~\ref{lemma: for PGL_3(F_q(t))}. Therefore, $\Bir(\PP_{\FF_q}^2 \times \PP_{\FF_q}^1)$ and $\Bir(\PP^n_{\FF_q})$ are both not Jordan.
\end{proof}

\section{Birational permutations}\label{sect: birational permutations}

In this section, we establish boundedness of finite subgroups in groups of birational permutations of varieties over finite fields proving Theorem~\ref{theo: BBir}. Let us recall the definitions. 

\begin{definition}\label{def: birational permutations}
    Let $X$ be a variety over a field $\KK$. For $f \in \Bir(X)$ denote by $\Ind(f)$ the locus of indeterminacy. Then the group of birational permutations $\BBir(X)$ is the following subgroup
    \[
    \BBir(X) = \{f \in \Bir(X) \mid \Ind(f)(\KK) = \Ind(f^{-1})(\KK) = \emptyset \}.
    \]
\end{definition}
   
\begin{definition}\label{def: p'-boundedness}
     For a prime number $p$ we say that finite group $G$ is a~$p'$-group (or a~$p'$-subgroup if $G$ is a subgroup of some other group) if~$\gcd(|G|,p) = 1$. We say that group $\Gamma$ has bounded finite $p'$-subgroups if there exists a constant $C$ such that for any finite~$p'$-subgroup~$\relpenalty=10000 G \subset \Gamma$ one has $|G| \leqslant C$.
\end{definition}

The following definition is due to Hu.

\begin{definition}[{\cite[Definition 1.6]{Hu20}}]\label{def: p-Jordan}
    Let $p$ be a prime number and $\Gamma$ be an infinite group. We say that $\Gamma$ is $p$-Jordan if there exist constants $J$ and $e$ such that any finite subgroup~$\relpenalty=10000 G \subset \Gamma$ contains a normal abelian subgroup $A$ with $[G:A] \leqslant J |G_{(p)}|^e$, where~$G_{(p)}\subset G$ is a $p$-Sylow subgroup.
\end{definition}

Let us prove a simple group-theoretic lemma.

\begin{lemma}\label{lemma: p'-boundedness => p-Jordan}
    Let $p$ be a prime number and $\Gamma$ be an infinite group. If $\Gamma$ has bounded~\mbox{$p'$-subgroups}, then $\Gamma$ is $p$-Jordan.
\end{lemma}

\begin{proof}
    By definition of boundedness of finite $p'$-subgroups, there exists a constant $C$ such that for any finite $p'$-subgroup $G \subset \Gamma$ one has $|G| \leqslant C$. Consider an arbitrary finite subgroup $G \subset \Gamma$. Denote by~$\mathfrak{L}$ the set of all primes different from $p$ and decompose the order of $G$ into a product of primes 
    \begin{equation}\label{eq: |G| = prod (q)}
    |G| = \left(\prod_{\ell \in \mathfrak{L}} \ell^{\alpha_{\ell}} \right) \cdot |G_{(p)}|,
    \end{equation}
    where $G_{(p)} \subset G$ is a $p$-Sylow subgroup. For each $\ell \in \mathfrak{L}$ such that $\alpha_{\ell} \neq 0$ consider a~\mbox{$\ell$-Sylow} subgroup $G_{(\ell)} \subset G$. Then $G_{(\ell)} \subset \Gamma$ is a finite $p'$-subgroup and $\ell^{\alpha_{\ell}} = |G_{(\ell)}| \leqslant C$. Therefore, the number of primes other than $p$ appearing in formula~\eqref{eq: |G| = prod (q)} is less than $C$ and we have 
    \[
    |G| =  \left(\prod_{\ell \in \mathfrak{L}} \ell^{\alpha_{\ell}} \right) \cdot |G_{(p)}| \leqslant C^C \cdot |G_{(p)}|.
    \]
    Hence, $\Gamma$ is $p$-Jordan according to Definition~\ref{def: p-Jordan}, take $J = C^C$, $e = 1$, and the trivial subgroup for~$A$.
\end{proof}

Before proving Theorem~\ref{theo: BBir}, let us recall the well known theorem about the action of a finite group on a tangent space to a fixed point.

\begin{theorem}[{see for example \cite[Theorem 3.7]{surfaces}}]\label{theo: action with a fixed point} Let $X$ be an irreducible algebraic variety over a field $\KK$ of characteristic $p$. Let $G$ be a finite group acting on $X$ with a fixed~$\KK$-point~$P$. Suppose
that $|G|$ is not divisible by $p$. Then the natural representation
\[
d\colon G \rightarrow \GL(T_P(X))
\]
is an embedding.
    
\end{theorem}

\begin{proof}[{of Theorem~\ref{theo: BBir}}]
    Denote by $N \geqslant 1$ the number of $\FF_q$-points on $X$ and denote one of these points by $P$. The group $\BBir(X)$ acts on $\FF_q$-points by permutations. Consider the stabilizer subgroup $\BBir_P(X) \subset \BBir(X)$ of $P$. Clearly, we have $$[\BBir(X):\BBir_P(X)] \leqslant N.$$ 

    Consider a finite $p'$-subgroup $G \subset \BBir(X)$. Consider the intersection $$\relpenalty=10000 \binoppenalty=10000 G_P = G \cap \BBir_P(X).$$ Then we have $[G:G_P] \leqslant N$ and $G_P$ is a finite $p'$-subgroup acting on $X$ with a fixed point~$P$. By Theorem~\ref{theo: action with a fixed point} we have $G_P \subset \GL(T_P(X)) \simeq \GL_d(\FF_q)$, where $d$ is the dimension of the tangent space at $P$ and 
    \[
    |G| = [G:G_P]\cdot|G_P| \leqslant N \cdot |\GL_d(\FF_q)|.
    \]
    This proves that $\BBir(X)$ has bounded finite $p'$-subgroups. Also, therefore, $\BBir(X)$ is~$p$-Jordan by Lemma~\ref{lemma: p'-boundedness => p-Jordan}. 
\end{proof}

\begin{rem}\label{rem: there is no boundedness of subgroups}
    In Theorem~\ref{theo: BBir}, the $p'$-boundedness cannot be strengthened to boundedness. For example, one can show that even the finite subgroups of $\BBir(\PP^1_{\FF_q} \times \PP^1_{\FF_q})$ are not bounded.
\end{rem}


\begin{thebibliography}{99}

%\bibitem{surfaces}
%\textbf{Y. Chen, C. Shramov.} Automorphisms of surfaces over fields of positive characteristic. https://arxiv.org/abs/2106.15906, (2021).

\bibitem{ALNZ22}
\textbf{S. Asgarli, K.-W. Lai, M. Nakahara, S. Zimmermann.}
Bijective Cremona transformations of the plane.
Sel. Math. New Ser. \textbf{28} (2022), article~53.

\bibitem{Can09}
\textbf{S. Cantat.}
Birational permutations.
C. R. Math. Acad. Sci. Paris \textbf{347} (2009), no.~21--22, 1289--1294.

\bibitem{surfaces}
\textbf{Y. Chen, C. Shramov.}
Automorphisms of surfaces over fields of positive characteristic.
Geom. Topol. \textbf{28} (2024), no.~6, 2747--2791.

\bibitem{DD}
\textbf{I. Dolgachev, A. Duncan.} Automorphisms of cubic surfaces in positive characteristic. Izv.
Math. \textbf{83} (2019), no. 3, 424--500.

\bibitem{D-I} 
\textbf{I. Dolgachev, V. Iskovskikh.} Finite subgroups of the plane Cremona group. In Algebra, arithmetic, and geometry: in honor of Yu. I. Manin, vol. I. Progr. Math., \textbf{269} (2009), 443--548.

\bibitem{Hu20}
\textbf{F. Hu.}
Jordan property for algebraic groups and automorphism groups of projective varieties
in arbitrary characteristic.
Indiana Univ. Math. J. \textbf{69} (2020), no.~7, 2493--2504.

\bibitem{LP11}
\textbf{M. J. Larsen, R. Pink.}
Finite subgroups of algebraic groups.
J. Amer. Math. Soc. \textbf{24} (2011), no.~4, 1105--1158.

%\bibitem{Pro21}
%\textbf{Y. Prokhorov.} Equivariant minimal model program. Russian Math. Surveys 76, no. 3, 461вЂ“542 (2021).

%\bibitem{main}
%\textbf{Y. Prokhorov, C. Shramov.} Jordan property for Cremona group over a finite field. https://arxiv.org/abs/2111.13367, (2021).

\bibitem{main}
\textbf{Y. Prokhorov, C. Shramov.}
Jordan property for the Cremona group over a finite field.
Proc. Steklov Inst. Math. \textbf{320} (2023), 278--289.

\bibitem{PrShr higher dim} 
\textbf{Y. Prokhorov, C. Shramov.} Jordan property for groups of birational selfmaps. Compositio Mathematica, \textbf{150} (2014), no. 12, 2054--2072.

\bibitem{Serre} \textbf{J.-P. Serre.} Le groupe de Cremona et ses sous-groupes finis. S\'eminaire Bourbaki, Nov. 2008, 61\textsuperscript{\`eme} ann\'e, \textbf{1000} (2010), 75--100.

\bibitem{Sil}
\textbf{J. H. Silverman.} The arithmetic of elliptic curves. Second edition. Graduate Texts in Mathematics, 106. Springer, Dordrecht (2009).

%\bibitem{Sti-book}
%\textbf{H. Stichtenoth.}
%Algebraic Function Fields and Codes.
%Second edition. Grad. Texts in Math. \textbf{254}, Springer, Berlin, 2009.

\bibitem{Sti73}
\textbf{H. Stichtenoth.} \"Uber die Automorphismengruppe eines algebraischen Funktionenk\"orpers
von Primzahlcharakteristik. I. Eine Absch\"atzung der Ordnung der Automorphismengruppe,
Arch. Math. (Basel) \textbf{24} (1973), 527--544. 

\bibitem{Vikulova} 
\textbf{A. Vikulova.} Jordan constant for the Cremona group of rank two over a finite field. Math. Notes, \textbf{113} (2023), no. 4, 587–592.

\bibitem{Weil}
\textbf{A. Weil.}
Sur les courbes alg\'ebriques et les vari\'et\'es qui s'en d\'eduisent.
Actualit\'es Sci. Ind. \textbf{1041},
Publ. Inst. Math. Univ. Strasbourg \textbf{7},
Hermann \& Cie, Paris, 1948.

\bibitem{Yas}
\textbf{E. Yasinsky.}
The Jordan constant for Cremona group of rank 2. Bulletin of the Korean Mathematical Society, \textbf{54} (2017), no. 5, 1859--1871.

\bibitem{Zai}
\textbf{A. Zaitsev.} Jordan constants of Cremona group of rank 2 over fields of characteristic zero. International Mathematics Research Notices, \textbf{2025} (2025), no. 6.




\end{thebibliography}
\end{document}